\documentclass[12pt,a4paper]{article}%
\usepackage{amsmath}
\usepackage[left=30mm,top=25mm,right=30mm,bottom=30mm]{geometry}
\usepackage{amsfonts}
\usepackage{amssymb}
\usepackage{pgf,tikz,pgfplots}
\usepackage{mathrsfs}
\usetikzlibrary{arrows}
\usepackage{amsthm}
\usepackage{color}
\usepackage[utf8]{inputenc}
\usepackage{graphicx}
\usepackage{enumerate}
\usepackage{hyperref}
\providecommand{\U}[1]{\protect\rule{.1in}{.1in}}
\newtheorem{theorem}{Theorem}
\newtheorem{lemma}{Lemma}[section]

\newtheorem{corollary}[theorem]{Corollary}
\newtheorem{definition}[lemma]{Definition}
\newtheorem*{definition*}{Definition}
\newtheorem{proposition}[lemma]{Proposition}
\renewenvironment{proof}[1][Proof]{\noindent\textbf{#1.} }{\ \rule{0.5em}{0.5em}}
\newenvironment{acknowledgement}{\textbf{Acknowledgement.}}{}

\theoremstyle{definition}
\newtheorem*{remark}{Remark}
\newtheorem*{question}{Question}

\newcommand{\Sch}{\mathrm{Sch}}

\newcommand{\N}{\mathbb{N}}

\newcommand{\F}{\mathcal{F}}
\newcommand{\G}{\mathcal{G}}
\newcommand{\GoG}{\mathfrak{G}}

\newcommand{\acts}{\curvearrowright}

\AtEndDocument{
\bigskip{
 \noindent
 \textsc{L\'{a}szl\'{o} M\'{a}rton T\'{o}th},
 \textsc{École polytechnique fédérale de Lausanne, Switzerland},
 \textit{E-mail}: \href{mailto:laszlomarton.toth@epfl.ch}{\texttt{laszlomarton.toth@epfl.ch}} 
}}

\begin{document}

\title{Invariant Schreier decorations of unimodular random networks}
\author{L\'{a}szl\'{o} M\'{a}rton T\'{o}th}
\date{}
\maketitle

\begin{abstract}
We prove that every $2d$-regular unimodular random network carries an invariant random Schreier decoration. Equivalently, it is the Schreier coset graph of an invariant random subgroup of the free group $F_d$. As a corollary we get that every $2d$-regular graphing is the local isomorphic image of a graphing coming from a p.m.p.\ action of $F_d$. 

The key ingredients of the analogous statement for finite graphs do not generalize verbatim to the measurable setting. We find a more subtle way of adapting these ingredients and prove measurable coloring theorems for graphings along the way.

\end{abstract}

\section{Introduction} \label{section:intro}

It is a nice exercise in combinatorics to show that any 2d-regular finite graph is a Schreier graph of the free group $F_d$ on $d$ generators. That is, the edges can be oriented and colored by $\{1,\ldots, d\}=[d]$ such that every vertex has exactly one incoming and outgoing edge of each color. We call such a coloring and orientation a \emph{Schreier decoration} of our graph. We do not want to spoil the argument for the reader, but will inevitably have to talk about parts of the solution. 
See \cite{cannizzo2013invariant} for a proof. 

The main result of this paper generalizes the above result to unimodular random graphs, a notion that arises naturally in the limit theory of sparse graph sequences. 

\begin{theorem}\label{theorem:main}
	Every $2d$-regular unimodular random rooted graph has an invariant random Schreier decoration.
\end{theorem}
\noindent 
We now introduce unimodularity and invariant random Schreier decorations. For the precise definitions omitted from the Introduction see Section \ref{section:preliminaries}.

Let $\GoG_{\circ}^{2d}$ denote the space of rooted, connected (possibly infinite) graphs with degrees bounded by $2d$. Similarly, let $\GoG_{\circ}^{\mathrm{Sch}}$ denote the space of rooted, connected Schreier graphs of $F_d$. There is a natural forgetting map $\Phi: \GoG_{\circ}^{\mathrm{Sch}} \to \GoG_{\circ}^{2d}$. 

We denote by $\mathcal{M}(X)$ the set of probability measures on a space $X$. A measure $\mu \in \mathcal{M}(\GoG_{\circ}^{2d})$ is called a \emph{unimodular random rooted graph} or \emph{unimodular random network} if it satisfies an involution invariance property mimicking reversibility of random walks. A \emph{random Schreier graph} is a measure $\mu' \in \mathcal{M}(\GoG_{\circ}^{\mathrm{Sch}})$. It is \emph{invariant}, if it is preserved by the natural action $F_d \acts \GoG_{\circ}^{\mathrm{Sch}}$, where generators $s \in [d]$ act by moving the root along the unique outward edge colored $s$ adjacent to it.  

\begin{definition} \label{definition:invariant_schreierization}
	Let $\mu \in \mathcal{M}(\GoG_{\circ}^{2d})$ be a unimodular random rooted graph that is $2d$-regular with probability 1. An invariant random Schreier decoration of $\mu$ is a measure $\mu' \in \mathcal{M}(\GoG_{\circ}^{\mathrm{Sch}})$ that is invariant and has $\Phi_{*} \mu' = \mu$. 
\end{definition}
\noindent
The condition $\Phi_{*} \mu' = \mu$ expresses that forgetting the orientation and coloring from a random Schreier graph generated by $\mu'$ gives the same random graph as $\mu$ in distribution.


\medskip

Cannizzo investigated the Schreier decorations of unimodular random rooted graphs in  \cite{cannizzo2013invariant}. He showed that for an invariant random Schreier graph $\mu'$ the pushforward $\Phi_{*} \mu'$ is always unimodular. Moreover, under some assumptions, if a unimodular random graph has an invariant random Schreier decoration, then it has uncountably many ergodic ones. Theorem \ref{theorem:main} fits in nicely with these results.


Invariant random Schreier graphs are closely related to Invariant Random Subgroups (IRS's) introduced by Abért, Glasner and Virág in \cite{abert2014kesten}. On one hand, the invariance property of the graph translates to conjugation invariance of the random subgroup that is the stabilizer of the root. On the other hand, given an IRS $H \leq F_d$, the Schreier graph on the coset space $F_d/H$ rooted at $H$ gives an invariant random Schreier graph. 

Keeping this correspondence in mind, Theorem \ref{theorem:main} can be thought of as a statement on IRS's of $F_d$. Bowen proved in \cite{bowen2012invariant} that $F_d$ has a large ``zoo'' of ergodic IRS's. Cannizzo's work shows that non-uni\-mod\-u\-lar\-i\-ty is a probabilistic obstruction to finding an IRS with a given factor geometry. Our result shows that it is the only obstruction. 

\medskip
The proof of the finite case relies on two parts. First, a finite $2d$-regular graph has a \emph{balanced} orientation. That is, an orientation with equal in- and outdegrees at all vertices. Second, a $d$-regular bipartite finite graph can be properly edge-colored with $d$ colors. This is Kőnig's line coloring theorem. Our approach to Theorem \ref{theorem:main} is to study these questions for \emph{graphings} instead of graphs.

Graphings are measurable, bounded degree graphs on some standard Borel probability space $(X,\nu)$ as vertex set with an additional measure preserving property. Their close connection to unimodular random graphs is also explained in Section \ref{section:preliminaries}. For now it suffices to know that any unimodular random graph $\mu$ can be encoded in a single graphing $\G$. We say $\G$ \emph{represents} $\mu$ if the connected component in $\G$ of a $\nu$-random vertex is the same as $\mu$ in distribution. Such a representing graphing exists for any unimodular random graph, though the choice is not unique.

Unfortunately neither of the two finite statements directly generalizes to graphings. In \cite{laczkovich1988closed} Laczkovich constructed a $2$-regular graphing that has no measurable balanced orientation. Moreover, there are $d$-regular bipartite graphings that have no measurable proper edge coloring with $d$ colors \cite[Section 6]{conley2013measurable}. We substitute these steps with the next two theorems.  

First we claim that representing graphings can be chosen smartly. 

\begin{theorem} \label{theorem:balanced_orientation}
	Let $\mu$ be a $2d$-regular unimodular random rooted graph. Then there exists a graphing $\G=(X,E,\nu)$ on a standard Borel space $(X, \nu)$ such that $\G$ represents $\mu$, and there is a measurable balanced orientation of the edges of $\G$.
\end{theorem}

Second, we claim that the coloring can be done with little error. Coloring questions for Borel graphs and graphings have been extensively studied, see \cite{kechris2015descriptive} for a survey. We build on the work of Csóka, Lippner and Pikhurko in \cite{csoka2016kHonig} to derive a version of Kőnig's theorem that suits our purposes.

\begin{theorem}\label{theorem:almost_proper_coloring}
	Let $\G=(X_1,X_2,E, \nu)$ be a $d$-regular bipartite graphing. For every $\varepsilon > 0$ there is a measurable proper edge coloring of $\G$ with $d+1$ colors, such that the measure of edges of the $(d+1)^{\mathrm{th}}$ color is at most $\varepsilon$.
\end{theorem}

\noindent
Combining Theorems \ref{theorem:balanced_orientation} and \ref{theorem:almost_proper_coloring} with a subsequential weak limit argument gives the proof of Theorem \ref{theorem:main}.  

\medskip

We can formulate corollaries of both Theorems \ref{theorem:main} and \ref{theorem:almost_proper_coloring} using local isomorphisms. We say a graphing $\G_1=(X_1, E_1, \nu_1)$ is the \emph{local isomorphic image} of another graphing $\G_2=(X_2, E_2, \nu_2)$ if their is a measure preserving map $\varphi: X_2 \to X_1$ that is a rooted isomorphism restricted to almost all the connected components of $\G_2$. 
It is clear that in this case $\G_1$ and $\G_2$ represent the same unimodular random graph.

\begin{corollary} \label{corollary:graphing_lift}
	Every $2d$-regular graphing $\G$ is the local isomorphic image of some graphing $\G'$ that comes from a p.m.p. action of $F_d$. That is, the edges of $\G'$ have a measurable balanced orientation and coloring with $[d]$.
\end{corollary}

\noindent
Note that Corollary \ref{corollary:graphing_lift} implies Theorem \ref{theorem:balanced_orientation} immediately. 

\begin{corollary} \label{corollary:proper_coloring_lift}
	Let $\G=(X_1,X_2,E)$ be a bipartite graphing with maximum degree $d$. Then $\G$ is the local isomorphic image of a bipartite graphing that can be measurably edge-colored with $d$ colors. 
\end{corollary} 

\noindent
Both corollaries are instances of a more general correspondence between coloring results for unimodular random networks and graphings. We express this phenomenon in Theorem \ref{theorem:translation} in Section \ref{section:graphings}.

\begin{remark}
	For graphings being bipartite is a stronger assumption than containing no odd cycles. This distinction is not relevant for our results, Theorem \ref{theorem:almost_proper_coloring} and Corollary \ref{corollary:proper_coloring_lift} hold with the weaker assumption as well. The reason for stating our results this way is that the relevant graphings arising in this paper are automatically bipartite. 
\end{remark}

It is natural to ask if any of these random decorations can be constructed as factor of i.i.d.\ (FIID) processes. See for example \cite{csoka2017invariant} for the definition of FIID for deterministic graphs, the definition for random graphs is analogous. It is not hard to see that the bi-infinite line has no FIID balanced orientation.

\begin{question}
	Which $2d$-regular unimodular random networks have a factor of i.i.d.\ Schreier decoration? Which even degree unimodular random networks have a factor of i.i.d.\ balanced orientations?
\end{question}

The structure of the paper is as follows. In Section \ref{section:preliminaries} we give a detailed introduction to unimodularity, graphings, and explain the connection to Benjamini-Schramm convergence. We also study the finite statement to motivate our question. We consider orientations in Section \ref{section:orientation} and prove Theorem \ref{theorem:balanced_orientation}. In Section \ref{section:coloring} we turn to coloring the edges and prove Theorem \ref{theorem:almost_proper_coloring}. We put these together and prove Theorem \ref{theorem:main} in Section \ref{section:proof_of_main}. In Section \ref{section:graphings} we prove Corollaries \ref{corollary:graphing_lift} and \ref{corollary:proper_coloring_lift}.

\medskip

\begin{acknowledgement}
The author would like to thank Miklós Abért, Damien Gaboriau, Ferenc Bencs and Nóra Szőke for inspiring conversations about the topic. 
\end{acknowledgement}

\section{Graphings, unimodularity and local convergence} \label{section:preliminaries}

In this section we will introduce the main notions of the paper and provide background and motivation for Theorem \ref{theorem:main}. A thorough and beautifully written exposition of the notions introduced here can be found in \cite{lovasz2012large}. 

Throughout this paper we will use standard letters $F,G, \ldots$ to denote countable graphs, calligraphic letters $\G, \mathcal{H}, \ldots$ to denote graphings and $\GoG$ with certain indices to denote spaces of graphs. The measures $\mu, \mu', \widetilde{\mu}, \ldots$ will denote random graphs, while $\nu, \nu', \widetilde{\nu}, \ldots$ will denote measures on vertex or edge sets of graphings. To avoid confusion, decorations of edges will be \emph{colors} and decorations of vertices will be \emph{labels}. 

The map $\Phi$ will always forget all the decoration of a graph. We do not burden the notation with indicating what type of decoration we are forgetting, this will always be clear from context. We consider Schreier graphs to be rooted and connected, unless explicitly stated otherwise. Properties of random graphs are always understood to hold with probability 1.

\subsection{Local convergence}  

Recall that $\GoG_{\circ}^{2d}$ denotes the space of rooted, connected graphs with degrees bounded by $2d$. For two such graphs $(G_1, o_1), (G_2, o_2) \in \GoG_{\circ}^{2d}$ we define their distance as
\[d\big((G_1, o_1), (G_2, o_2)\big) = \frac{1}{2^r},\] 
where $r$ is the largest integer with $B_{G_1}(r,o_1) \cong B_{G_2}(r,o_2)$. $B_G(r,o)$ denotes the rooted ball in $G$ around $o$ of radius $r$, the isomorphism is understood as rooted graphs. Two rooted graphs are close, if they look the same in a large neighborhood of the root. If the balls above are isomorphic for all $r \in \N$, then $(G_1,o_1)$ and $(G_2,o_2)$ are isomorphic, and their distance is defined 0. With this distance $\GoG_{\circ}^{2d}$ becomes a totally disconnected compact metric space.

Local (or Benjamini-Schramm \cite{benjamini2001}) convergence takes a connected finite graph $F$ and by choosing a uniform random root among its vertices turns it into the random rooted graph $(F,o)$. If the graph is not connected, only the component of the random root is kept. This way we obtain a probability measure $\mu_F$ on $\GoG_{\circ}^{2d}$. Formally we define the map $f:V(F) \to \GoG_{\circ}^{2d}$, $f(v) = \big(C_F(v),v\big)$, where $C_F(v)$ denotes the component of $F$ containing $v$. Then $\mu_F$ is simply the pushforward of the uniform measure on $V$ onto $\GoG_{\circ}^{2d}$ by $f$.  

A sequence of graphs $(F_n)$ is convergent, if the corresponding $\mu_{F_n}$ converge weakly in the space $\mathcal{M}(\GoG_{\circ}^{2d})$ of probability measures on $\GoG_{\circ}^{2d}$. In this case the weak limit $\mu$ is considered to be the limit of the finite graphs $F_n$. Since $\GoG_{\circ}^{2d}$ is totally disconnected, weak convergence is equivalent to the convergence of the measures of the following clopen sets. For a finite rooted graph $(F,v) \in \GoG_{\circ}^{2d}$ of radius $r$ we define the clopen set

\[\GoG_{(F,v)}=\big\{(G,o) \in \GoG_{\circ}^{2d} ~|~ B_{G}(r,o) \cong (F,v)\big\}.\]
A sequence of measures $\mu_n \in \mathcal{M}(\GoG_{\circ}^{2d})$ converges weakly to $\mu$ if and only if $\mu_n(\GoG_{(F,v)}) \to \mu(\GoG_{(F,v)})$ for all choices of $(F,v)$. 

While the $\mu_{F_n}$ are supported on finite rooted graphs, their limit $\mu$ can be supported on infinite graphs as well. 

\subsection{Sofic and unimodular measures}

A central question of this limit theory is to determine which measures are limits of finite graphs. This leads to the notions of \emph{sofic} and \emph{unimodular} measures. 

Soficity is easy to define: a measure $\mu \in \mathcal{M}(\GoG_{\circ}^{2d})$ is sofic if it is the limit of measures $\mu_{F_n}$ corresponding to finite graphs. 

Unimodularity is more subtle, and we will only define it in the regular case. See for example \cite{lovasz2012large} for a general treatment. Let $\mu \in \mathcal{M}(\GoG_{\circ}^{2d})$ be $2d$-regular with probability 1. It is unimodular if it satisfies the following involution invariance property. We generate a random instance $(G,o)$ of $\mu$. We then choose a uniform random neighbor $o'$ of $o$ in $G$. We say that $\mu$ is unimodular, if the random birooted graph $(G,o,o')$ is invariant with respect to inverting the order of the roots.

Formally we introduce the space $\GoG_{\circ\circ}^{2d}$ of connected graphs with degrees bounded by $2d$, that have two distinguished vertices, the first and the second root. The topology is defined similarly to $\GoG_{\circ}^{2d}$, two such birooted graphs are close, if the look the same in large neighborhoods around both roots. There is a natural involution $\iota: \GoG_{\circ\circ}^{2d} \to \GoG_{\circ\circ}^{2d}$ swapping the roots: $\iota (G,o,o') = (G,o',o)$. The random rooted graph $\mu \in \mathcal{M}(\GoG_{\circ}^{2d})$ determines a measure $\widetilde{\mu} \in \mathcal{M}(\GoG_{\circ\circ}^{2d})$ as above, by choosing the second root uniformly among the neighbors of the original. We say $\mu$ is unimodular if $\iota_{*}\widetilde{\mu}=\widetilde{\mu}$. 

We mention two important facts about unimodularity. First, measures $\mu_F$ coming from finite graphs are unimodular. This is easy to check, it is just a reformulation of the simple random walk being not only stationary, but also reversible with respect to the uniform measure on the vertices of a finite regular graph.

Second, unimodularity is a closed condition in $\mathcal{M}(\GoG_{\circ}^{2d})$ with respect to weak convergence. The reason is that involution invariance can be checked using only the local neighborhood probabilities $\mu(\GoG_{(F,v)})$.

These two facts imply that every sofic measure is unimodular. The converse however is an open question, it is not known whether there is a non-sofic unimodular measure. 

\subsection{Invariance in the finite case}

To motivate Theorem \ref{theorem:main} we explore how the invariance property appears in the finite case. 


Our finite statement takes a connected finite graph $F$ and turns it into an unrooted Schreier graph $\widetilde{F}=(F, \mathrm{or}, c)$, where $\mathrm{or}$ and $c$ signify the orientation and coloring respectively. If we then choose a uniform random root, we get the random rooted Schreier graph $(F, o, \mathrm{or}, c)$, where $o \in V(F)$ is uniform random. We denote its distribution by $\mu'_{\widetilde{F}} \in \mathcal{M}(\GoG_{\circ}^{\mathrm{Sch}})$.
The action of any generator $s$ is by a bijection on $V(F)$, so the uniform distribution on $V(F)$ is invariant with respect to the action $F_d \acts V(F)$. This implies that $\mu'_{\widetilde{F}}$ is invariant.

The graph structure of $\mu'_{\widetilde{F}}$ without the orientation and coloring is the same as $\mu_F$. Using our notation, we say $\Phi_{*} \mu'_{\widetilde{F}} = \mu_F$. These two phenomena, $\mu'_{\widetilde{F}}$ being invariant and $\Phi_{*} \mu'_{\widetilde{F}} = \mu_F$ are what motivate the definition of an invariant random Schreier decorations.

With this in mind it is relatively easy to prove that sofic random rooted graphs have invariant random Schreier decorations. Finding Schreier decorations of the finite approximations, then taking a subsequential weak limit of the corresponding random rooted Schreier graphs solves the problem. Theorem \ref{theorem:main} shows that the property of having an invariant random Schreier decoration does not distinguish soficity and unimodularity.

\begin{remark}
	For a fixed connected, $2d$-regular infinite rooted or unrooted graph finding a Schreier decoration is simply a compactness argument once one knows the finite version. In fact rooted (possibly infinite), connected Schreier graphs are in one-to-one correspondence with subgroups of $F_d$. 
	
	If one wants to build an invariant Schreier decoration however, one has to pick the individual Schreier decorations of the random instances of the unimodular random graph in a coherent way, so that the whole process becomes invariant. The additional freedom is that the choice can be random.
\end{remark}

\subsection{Graphings}

Graphings have quite a few equivalent definitions. We are going to stick with \cite{lovasz2012large}. 

\begin{definition}[Graphing] \label{def:graphing}
	Let $X$ be a standard Borel space and let $\nu$ be a probability measure on the Borel sets in $X$. A \emph{graphing} (with degree bound $2d$) is a graph $\G$ on $V(\G) = X $ with Borel edge set $E(\G) \subset X \times X$ in which all degrees are at most $2d$ and 
	\begin{equation} \label{eqn:graphing}
	\int_{A} \deg(x,B) \ d \nu (x) = \int_{B} \deg(x,A) \ d \nu (x)
	\end{equation}
	for all measurable sets $A,B \subseteq X$, where $\deg(x,S)$ is the number of edges from $x \in X$ to $S \subseteq X$.
\end{definition}

Formally the graphing consists of the quadruple $\G=(X,\mathcal{B}, E, \nu)$, where $\mathcal{B}$ denotes the Borel $\sigma$-algebra on $X$. We will suppress $\mathcal{B}$ from the notation.



Looking at the connected component of a random vertex will represent a  unimodular random rooted graph. Recall that for $x\in X$ its connected component in $\G$ is denoted $C_{\G}(x)$. It is a basic property of Borel graphs that the map $f : X \to \GoG_{\circ}^{2d}$, defined by $f(x)=\big(C_{\G}(x),x\big)$ is Borel. This way we can define $\mu_{\G} = f_{*} \nu$ similarly to the finite graph case.

The measure preserving property (\ref{eqn:graphing}) is closely related to the involution invariance that we used to define unimodularity. In fact (\ref{eqn:graphing}) implies that $\mu_{\G}$ is unimodular.

In the other direction the connection is a bit more subtle. Fix a unimodular random rooted graph $\mu \in \mathcal{M}(\GoG_{\circ}^{2d})$. We say that a graphing $\G$ represents $\mu$ if $\mu=\mu_{\G}$. 

\begin{theorem}[\cite{aldous2007processes}, \cite{elek2007note}] \label{theorem:representing_graphing}
	Every unimodular random rooted graph is represented by some graphing.
\end{theorem}

As mentioned in the Introduction, the choice of a representing graphing is not unique. Graphings with distinctly different global measurable structure can represent the same unimodular random rooted graph.

In \cite{lovasz2012large} a construction is carefully explained and motivated. The representing graphings described there always have the same vertex and edge sets, only the measure changes. The space is a vertex-labeled version of $\GoG_{\circ}^{2d}$, with edges connecting graphs that are the same up to a displacement of the root to a neighbor. 

Starting from a unimodular random rooted graph $\mu$ one uses further random vertex labels to break all potential symmetries of any realizations. This results in a measure $\nu$ on the above space. The unimodularity of $\mu$ implies that (\ref{eqn:graphing}) is satisfied by $\nu$.

\subsection{Measuring edges of graphings}

The property \ref{eqn:graphing} of graphings allows one to measure set of edges in a meaningful way. Let $\widetilde{\nu}$ be the measure on $E \subseteq X \times X$ obtained the following way. For a measurable subset $C \subseteq E$ of edges let 
\[\widetilde{\nu} (C) = \frac{1}{2} \int_{X} \deg_{C}(x) \ d \nu(x).\]

Here $\deg_{C}(x)$ denotes the degree of $x$ using the edge set $C$, that is the number of edges in $C$ adjacent to $x$. The normalization accounts for counting edges from both endpoints. Intuitively equation \ref{eqn:graphing} expresses that edges have the same weight from both endpoints. For the edge set $C=E(A,B)$ of edges between $A$ and $B$ it say that measuring them from $A$ and $B$ gives the same answer. In our case the measure $\widetilde{\nu}$ takes values in $[0,2d]$.

\subsection{Sparse vertex labeling of graphings}

We will use vertex labeling to break local symmetries of graphings. Corollary \ref{corollary:sparse_labeling} below is a basic tool when working with graphings.

We say a labeling of a graph is $r$-sparse, if vertices of distance at most $r$ always have distinct labels. A $1$-sparse labeling is traditionally called a \emph{proper} labeling. 

Concerning vertex labelings of Borel graphs Kechris, Solecki and Todorcevic \cite[Theorem 4.6]{kechris1999borel} show that if the degree is bounded by $\Delta$, then there is a proper Borel labeling of the vertices with $\Delta+1$ labels.

One can also show that introducing additional edges in a Borel graph connecting vertices at distance at most $r$ gives a Borel graph. This together with the above coloring result has the following useful corollary. 

\begin{proposition}\label{corollary:sparse_labeling}
	For every bounded degree Borel graph $\G$ and $r \in \N$ there is an $r$-sparse Borel labeling $l : V \to [k]$ for some $k \in \N$.
\end{proposition}
\noindent
See for example \cite[Corollary 3.1]{csoka2016kHonig} for a more detailed exposition.

\section{Orientation} \label{section:orientation}

In this section we prove Theorem \ref{theorem:balanced_orientation}. At first we exhibit a way to choose a random balanced orientation of rooted, even degree infinite graphs. The important property will be that for any even degree graph $G$ and vertices $v_1, v_2$ the random orientations of $(G, v_1)$ and $(G,v_2)$ will be the same in distribution on the edges of $G$. We then explain how this allows one to construct a measurably orientable graphing to represent a unimodular random rooted graph. 

\subsection{Orienting even degree trees} \label{subsection:orienting_trees}
Let $(T,o)$ be an infinite (connected) rooted tree with even degrees. We describe a random balanced orientation of $(T,o)$. For all vertices $v \in V(T)$ let $\deg(v)=2d_v$. We start at the root, and orient $d_o$ edges inwards and $d_o$ edges outwards randomly. We proceed to the neighbors of $o$. If $v$ is a neighbor of $o$, then one of the $2 d_v$ edges at $v$ is already oriented, say towards $v$. So from the $2d_v-1$ undirected ones we orient $d_v-1$ inwards and $d_v$ outwards randomly, independently of the previous choice. We continue this procedure radially outwards from $o$, independently randomizing at each vertex. As $T$ is a tree we will always have exactly 1 already oriented edge at each vertex. The result of this random procedure is a balanced orientation of $(T,o)$.

\begin{lemma}
	The random orientation described above does not depend on the choice of root in $T$. More precisely, if $v_1$ and $v_2$ are two vertices of $T$, then the random orientations started at $v_1$ and $v_2$ are the same in distribution.
\end{lemma}

\begin{proof}
	Pick any \emph{saturated} finite subtree $F \subseteq T$ containing both $v_1$ and $v_2$ as internal vertices. By saturated we mean that vertices of $F$ are either leaves of $F$, or have all their neighbors in $F$ as well. The internal vertices of $F$ are the non-leaves. Fix any orientation of the edges of $F$ that is balanced at internal vertices. We compute the probability of getting this orientation of $F$ when starting our procedure at $v_1$. List the internal vertices of $F$ in a radially expanding order: $\{u_1, u_2, \ldots, u_n\}$, specifically $u_1 = v_1$. At $u_1$ we have to match 1 out of ${2d_{u_1} \choose d_{u_1}}$ possible orientations. At each later $u_i$ we have to match 1 out of the possible ${2d_{u_i}-1 \choose d_{u_i}}$ orientations of the remaining edges. Note that this does not depend on whether the edge already oriented at $u_i$ is inwards or outwards. In total, the probability of generating our fixed orientation from $v_1$ is
	
	\[\frac{1}{{2d_{u_1} \choose d_{u_1}}}\prod_{i=2}^{n}\frac{1}{{2d_{u_i}-1 \choose d_{u_i}}}.\]
	
	Using that ${2d_{u_1} \choose d_{u_1}} = 2 \cdot {2d_{u_1}-1 \choose d_{u_1}}$, we can rewrite this as 
	
	\[\frac{1}{2}\prod_{i=1}^{n}\frac{1}{{2d_{u_i}-1 \choose d_{u_i}}}.\]
	
	Notice that this later expression does not change if we reorder the $u_i$, so starting from $v_2$ we will get the same probability.
	
	This holds for all $F$ as above, and this implies that the two random processes are the same in distribution. Cylinders over such finite subgraphs form a base of the topology on the set of orientations.
\end{proof}

\subsection{Orienting even degree graphs}
We now define a random balanced orientation of rooted, bounded even degree infinite graphs. Again, the random orientation will not depend on the choice of root in distribution. 

Given such a graph $(G,o)$, we will first orient cycles as long as there are any. Our procedure will consist of countably many stages, with each stage consisting of countably many steps. During the procedure, we will have oriented and undirected edges. We say that at a certain point a cycle is undirected, if all its edges are undirected. Two cycles of the same length are neighbors, if they share an edge.

After stage $n$, with probability 1 there will be no undirected cycles of length at most $n$. At every step of stage $n$ we consider currently undirected cycles of length $n$, and randomize a uniform $[0,1]$ label for each. If an undirected cycle has a higher label than all of its neighbors, then it gets oriented, randomly in one of the two possible directions.

Since our graph has bounded degree, the number of neighbors of a cycle considered during stage $n$ is bounded. So as long as it remains undirected, it has a positive probability (bounded away from 0) to get a higher label than all its neighbors at every step. Therefore staying undirected after countably many steps happens with probability 0.

At every step the partial orientation is balanced, the indegrees and outdegrees of vertices are the same. Clearly the undirected degree of each vertex is always even. By the end of our procedure with probability 1 the undirected edges contain no cycle, so they span a disjoint union of even degree trees. We can randomly orient these trees by choosing a root in each component: for example the closest one to $o$ in $G$, or if there are multiple ones, then picking one uniformly at random. Then we use the random orientation described in Subsection \ref{subsection:orienting_trees} for the rooted trees.

We call the random balanced orientation of $(G,o)$ described above the \emph{canonical} random balanced orientation of $(G,o)$. The distribution on the edges does not depend on the choice of a root, as the cycle removal never even considered the root, while the orientation of the trees was shown to be invariant under changing the root. In the next subsection we show how the canonical random orientation allows us to build a measurably orientable graphing for any unimodular random network.

\subsection{Constructing orientable graphings} \label{subsection:constructing_graphings}

Theorem \ref{theorem:representing_graphing} states that any unimodular random network can be represented by a graphing. The construction thoroughly explained in \cite[Theorem 18.37]{lovasz2012large} can be adapted to our situation. We outline our construction and the key ideas, but do not burden the reader with lengthy exposition. 

The set of vertices of our graphing will be the space of $2d$-regular, oriented, connected rooted graphs with labels from $[0,1]$ on each vertex. That is, every node is a $2d$-regular rooted graph $(H,v)$ with an orientation $\textrm{or}: E(H) \to V(H)$ and a vertex labeling $l: V(H) \to [0,1]$. Let $\mathfrak{G}^{\textrm{or},l}_{\circ}$ denote the set of such quadruples.

The Borel structure is generated by the following cylinder sets: for any $r \geq 0$, we fix the oriented isomorphism type of the ball with radius $r$ about the root, and also for every vertex in the ball we specify a Borel set in $[0, 1]$ from which the weight is to be chosen.

For a node $(H,v,\textrm{or}, l)$ we introduce an edge to the nodes $(H,v',\textrm{or}, l)$, where $v'$ runs through the neighbors of $v$ in $H$.
In other words, $(H,v,\textrm{or}, l)$ and $(H',v',\textrm{or}', l')$ are connected if there is an isomorphism (of unrooted graphs) $\eta: H \to H'$, that preserves the orientation and the vertex labels, and $\eta(v) \sim v'$ in $H'$.

Starting from a unimodular measure $\mu$ on $\mathfrak{G}_{\circ}^{2d}$ we construct a measure $\bar{\mu}$ on $\mathfrak{G}^{\textrm{or},l}_{\circ}$. We first choose $(G,o)$ according to $\mu$, then put i.i.d.\ $[0,1]$ labels on the vertices, and randomize an orientation according to the canonical balanced orientation of $(G,o)$. 

\begin{proposition} \label{proposition:orientable_graphing}
	The Borel graph $\mathfrak{G}^{\textrm{or},l}_{\circ}$ with the measure $\bar{\mu}$ represents the unimodular random rooted graph $\mu$. Moreover, it has a measurable balanced orientation. 
\end{proposition}

\begin{proof}
	The measure $\bar{\mu}$ will be involution invariant on $\mathfrak{G}^{\textrm{or},l}_{\circ}$, because $\mu$ is involution invariant on $\mathfrak{G}_{\circ}^{2d}$, and both the orientation an the labels are chosen in a way that is in distribution the same regardless of the position of the root. This makes sure that $(\mathfrak{G}^{\textrm{or},l}_{\circ}, \bar{\mu})$ is indeed a graphing. 
	
	The labeling breaks all symmetries of $(G,o, \mathrm{or})$ almost surely, which makes sure that $(\mathfrak{G}^{\textrm{or},l}_{\circ}, \bar{\mu})$ represents $\mu$.
	
	Finally by construction the edges of $\mathfrak{G}^{\textrm{or},l}_{\circ}$ can be measurably oriented. The edge $(H,v,\textrm{or}, l) \sim (H,v',\textrm{or}, l)$ is oriented towards $(H,v,\textrm{or}, l)$ if and only if $(v,v')$ is oriented towards $v$ in $(H,v,\textrm{or})$.
\end{proof}

\medskip

This finishes the proof of Theorem \ref{theorem:balanced_orientation}.

\section{Measurable colorings of bipartite graphings} \label{section:coloring}

In this section we prove Theorem \ref{theorem:almost_proper_coloring}. We will repeatedly make use of two results of Csóka, Lippner and Pikhurko \cite[Theorems 1.9 and 4.2]{csoka2016kHonig}. Their main result is the following

\begin{theorem} \label{theorem:CSLP_main}
	For every $d \geq 1$ there is $r_0 = r_0(d)$ such that if $\G = (X_1, X_2,E, \nu)$ is a bipartite graphing with maximum degree at most $d + 1$ such that the set $J$ of vertices of degree $d + 1$ is $r_0$-sparse, then there is a measurable proper edge coloring of $\G$ with $d+1$ colors.
\end{theorem}

During the proof of \ref{theorem:CSLP_main} they use the following theorem as an inductive step.

\begin{theorem} \label{theorem:CSLP_induction}
	For every $d > 2$ and $r > 1$, there is $r_1 = r_1(d, r)$ such that the
	following holds. Let $\G = (X_1,X_2, E, \nu)$ be a bipartite graphing with degree bound $d + 1$ that has no finite components. If the set $J \subseteq X_1 \cup X_2$ of vertices of degree $d + 1$ is $r_1$-sparse, then there is a Borel matching $M$ such that, up to removing a null-set, $G \setminus M$ has maximum degree at most $d$ and its set of degree-$d$ vertices is $r$-sparse.
\end{theorem}

We will follow the induction and make sure the color $(d+1)$ is used rarely. Note that vertices of high degree being sparse does not immediately imply that they have small density, there might be small connected components. However, if a component contains at least 2 such vertices, then the density can be bound in terms of the sparsity. Ruling out components with exactly one high degree vertex is key to our proof of Theorem \ref{theorem:almost_proper_coloring}.

\medskip
We start by proving two lemmas to take care of finite components in graphings.

\begin{lemma} \label{lemma:finite_smart_coloring}
	Let $G=(S,T,E)$ be a finite bipartite graph with degrees bounded by $(d+1)$, and such that 
	\begin{enumerate}
		\item \label{condition:bad_points_sparse} the vertices of degree $(d+1)$ are 3-sparse, and
		\item there are at most $\rho \cdot |V(G)|$ vertices of degree $(d+1)$ in total, $(0 \leq \rho \leq 1)$.
	\end{enumerate} 
	Then there is a proper edge coloring of $G$ with $(d+1)$ colors such that there are at most $\rho \cdot |V(G)|$ edges of color $(d+1)$.
\end{lemma}

\begin{proof}
	For each vertex of degree $(d+1)$ choose an arbitrary adjacent edge and color it with $(d+1)$. As these points are 3-sparse, this coloring is so far proper. We colored at most $\rho \cdot |V(G)|$ edges with color $(d+1)$. The remaining graph has degrees at most $d$, so by Kőnig's line coloring theorem it can be properly colored with $d$ colors. Notice that if $\rho=0$ the statement is simply Kőnig's theorem. 
\end{proof}

\medskip

\begin{lemma} \label{lemma:finite_components_smart_coloring}
	Let $\G=(X_1,X_2,E, \nu)$ be a bipartite graphing with all connected components finite and satisfying the conditions of Lemma \ref{lemma:finite_smart_coloring}. Then it has a measurable proper edge coloring with $(d+1)$ colors such that the $\widetilde{\nu}$-measure of edges of color $(d+1)$ is at most $\rho$.
\end{lemma}

\begin{proof}
	For $i=1,2,...$ we edge color the components with exactly $i+1$ vertices. For each $i$ we fix an $i$-sparse Borel labeling $l: X_1 \cup X_2 \to [k_i]$. In components of size $i+1$ all vertices have different label. For each isomorphism type of such labeled graphs choose a way to color its edges as in Lemma \ref{lemma:finite_smart_coloring}, and apply this coloring consistently everywhere. At each step the color classes are Borel, and there are countably many steps, so the color classes in the end are also Borel.
	 
	At each step the measure of edges colored by $(d+1)$ is at most $\rho$ times the measure of vertices with components of size $i+1$. As $\mu(X_1 \cup X_2) = 1$, the measure of edges colored $(d+1)$ in the end is at most $\rho$. 
\end{proof}

\medskip
The following proposition is the first step of our inductive procedure to prove Theorem \ref{theorem:almost_proper_coloring}.

\begin{proposition} \label{proposition:induction_start}
	Let $\G=(X_1,X_2,E, \nu)$ be a bipartite graphing with degrees $2$ or $3$. Assume that the vertices of degree $3$ are $r$-sparse for $r \geq r_0(3)$, where $r_0$ is the constant from Theorem \ref{theorem:CSLP_main}. Then it has a measurable edge coloring with 3 colors such that the $\widetilde{nu}$-measure of edges of the 3rd color is at most $4/r$. 
\end{proposition}

\begin{proof}
	We will use 3 colors: red, blue and purple, with as few purple edges as possible. It suffices to prove the theorem when $\G$ only has finite components, and also when it only has infinite components. In general $\G$ can be split into two measurable parts according to the components being finite or infinite, and measurably coloring each side with few purple edges.
	
	\smallskip

	\textbf{Suppose all components are finite.} Finite components that contain no vertices of degree 3 are even cycles, and can be properly edge-colored with red and blue, no purple edges needed. 
	
	We also claim that finite components that contain a degree 3 vertex contain at least two. Indeed, the sum of degrees has to be even. This implies, that the size of such a finite component is at least $r+1$, as degree 3 vertices are $r$-sparse. We claim that the density of degree 3 edges is at most $2/r$. Indeed, because of the sparsity the balls of radius $r/2$ around the degree 3 points are disjoint, and these balls contain at least $r/2$ points each. By Lemma \ref{lemma:finite_smart_coloring} we can color these finite components with only $2/r$ proportion of edges colored purple.
	
	As in the proof of Lemma \ref{lemma:finite_components_smart_coloring} we can measurably do this over all finite components of $\G$, and we will end up with at most $\rho$ purple edges in measure.
	
	\smallskip
	
	\textbf{Now assume that all components are infinite.} We simply apply Theorem \ref{theorem:CSLP_main}. We get a measurable edge coloring $c:E \to \{\mathrm{red}, \mathrm{blue}, \mathrm{purple}\}$. We will modify this coloring in finitely many steps so that purple edges become sparse. 
	
	Call a purple edge \emph{standard}, if both of its endpoints have degree 2. Our procedure will consist of $r+1$ phases, each phase consisting of finitely many recoloring steps. In phase 0 we will make sure that the neighbors of standard purple edges have different color. Standard purple edges that are between two blue edges can be recolored red. We will now argue that this recoloring maintains measurability.
	
	Fix a Borel 4-sparse labeling $l:X_1 \cup X_2 \to [k]$. Let $A=\{x \in X_1 \cup X_2 ~|~ B_{\G}(x,2) \cong \alpha\}$, where $\alpha$ is the rooted, edge-colored and vertex-labeled graph in Figure \ref{figure:neighborhood1}.
	
	\definecolor{qqqqff}{rgb}{0.,0.,1.}
	\definecolor{yqqqyq}{rgb}{0.5019607843137255,0.,0.5019607843137255}

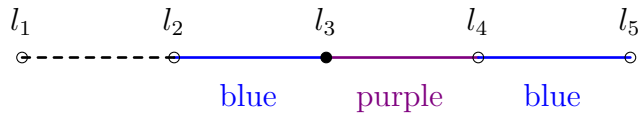
\begin{figure}[h]
	\centering
	\begin{tikzpicture}[line cap=round,line join=round,>=triangle 45,x=1.0cm,y=1.0cm]
	\clip(-1,-1) rectangle (9.,1.3);
	\draw [line width=1.pt,color=yqqqyq] (4.,0.)-- (6.,0.);
	\draw [line width=1.pt,color=qqqqff] (6.,0.)-- (8.,0.);
	\draw [line width=1.pt,color=qqqqff] (4.,0.)-- (2.,0.);
	\draw [line width=1.pt,dash pattern=on 3pt off 3pt] (2.,0.)-- (0.,0.);
	\draw (-0.3,0.8) node[anchor=north west] {$l_1$};
	\draw (1.7,0.8) node[anchor=north west] {$l_2$};
	\draw (3.7,0.8) node[anchor=north west] {$l_3$};
	\draw (5.7,0.8) node[anchor=north west] {$l_4$};
	\draw (7.7,0.8) node[anchor=north west] {$l_5$};
	\draw [color=qqqqff](2.45,-0.2) node[anchor=north west] {\textrm{blue}};
	\draw [color=yqqqyq](4.25,-0.2) node[anchor=north west] {\textrm{purple}};
	\draw [color=qqqqff](6.45,-0.2) node[anchor=north west] {\textrm{blue}};
	\begin{scriptsize}
	\draw [fill=black] (4.,0.) circle (2pt);
	\draw [color=black] (6.,0.) circle (2pt);
	\draw [color=black] (8.,0.) circle (2pt);
	\draw [color=black] (2.,0.) circle (2pt);
	\draw [color=black] (0.,0.) circle (2pt);
	\end{scriptsize}
	\end{tikzpicture}
	\caption{The rooted, edge-colored and vertex-labeled graph $\alpha$.}
	\label{figure:neighborhood1}	
\end{figure}

	 
	The vertex labels $l_1, \ldots, l_5$ and the color of the left most edge are not specified, but fix them for the time being. For any point $x \in A$ there exists a unique (standard) purple edge $e_a \in E$ starting at $a$. The vertex labeling is sparse so if $a, b \in A$ are distinct, then  $e_a$ and $e_b$ are also distinct. Since both the vertex labeling and the edge coloring are measurable, the set of edges $E_A=\{e_a ~|~ a \in A\}$  is measurable. We change the color of these to red.
	
	Now we repeat this procedure for all possible choices of $l_1, \ldots, l_5$ one after the other. We also consider neighborhoods where the vertex labeled by $l_2$ has degree 3, so the neighborhood is like in Figure \ref{figure:neighborhood2}.
	
	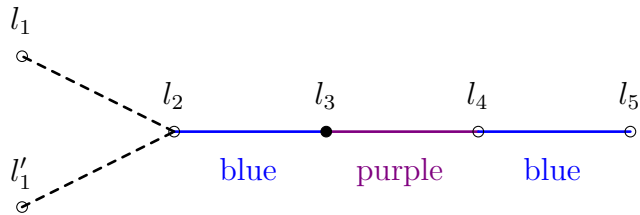
\begin{figure}[h]
		\centering
		\definecolor{qqqqff}{rgb}{0.,0.,1.}
		\definecolor{yqqqyq}{rgb}{0.5019607843137255,0.,0.5019607843137255}
		\begin{tikzpicture}[line cap=round,line join=round,>=triangle 45,x=1.0cm,y=1.0cm]
		\clip(-1.,-2.) rectangle (9.,2.3);
		\draw [line width=1.pt,color=yqqqyq] (4.,0.)-- (6.,0.);
		\draw [line width=1.pt,color=qqqqff] (6.,0.)-- (8.,0.);
		\draw [line width=1.pt,color=qqqqff] (4.,0.)-- (2.,0.);
		\draw [line width=1.pt,dash pattern=on 3pt off 3pt] (2.,0.)-- (0.,1.);
		\draw [line width=1.pt,dash pattern=on 3pt off 3pt] (0.,-1.)-- (2.,0.);
		\draw (-0.3,1.8) node[anchor=north west] {$l_1$};
		\draw (-0.3,-0.2) node[anchor=north west] {$l'_1$};
		\draw (1.7,0.8) node[anchor=north west] {$l_2$};
		\draw (3.7,0.8) node[anchor=north west] {$l_3$};
		\draw (5.7,0.8) node[anchor=north west] {$l_4$};
		\draw (7.7,0.8) node[anchor=north west] {$l_5$};
		\draw [color=qqqqff](2.45,-0.2) node[anchor=north west] {\textrm{blue}};
		\draw [color=yqqqyq](4.25,-0.2) node[anchor=north west] {\textrm{purple}};
		\draw [color=qqqqff](6.45,-0.2) node[anchor=north west] {\textrm{blue}};
		\begin{scriptsize}
		\draw [fill=black] (4.,0.) circle (2pt);
		\draw [color=black] (6.,0.) circle (2pt);
		\draw [color=black] (8.,0.) circle (2pt);
		\draw [color=black] (2.,0.) circle (2pt);
		\draw [color=black] (0.,1.) circle (2pt);
		\draw [color=black] (0.,-1.) circle (2pt);
		\end{scriptsize}
		\end{tikzpicture}
		\caption{Other possible neighborhoods of purple edges between blue edges.}
		\label{figure:neighborhood2}
	\end{figure}
	 
	After these finitely many recoloring steps there will be no standard purple edges between two blue edges. Endpoints of such purple edges can be identified by looking at their 2-neighborhoods, and the sparse labeling makes sure that we do not try to recolor an edge twice during the same step. By the same argument we can make sure that there are no standard purple edges between two red edges. This finishes phase 0.  
	
	Generally for $n\geq 1$, during phase $n$ we make sure that there are no standard purple edges at distance at most $n$. Since phases 1 through $n-1$ have already finished, a path between two purple edges at distance $n$ is colored by red and blue in an alternating way. Phase 0 has made sure that the neighbors of our purple edges are colored differently.
	
	\begin{figure}[h]
		\centering
		\definecolor{yqqqyq}{rgb}{0.5019607843137255,0.,0.5019607843137255}
		\definecolor{ffqqqq}{rgb}{1.,0.,0.}
		\definecolor{qqqqff}{rgb}{0.,0.,1.}
		\begin{tikzpicture}[line cap=round,line join=round,>=triangle 45,x=1.0cm,y=1.0cm]
		\clip(-0.5,-1.) rectangle (12.5,1.);
		\draw [line width=1.pt,color=qqqqff] (4.,0.)-- (6.,0.);
		\draw [line width=1.pt,color=ffqqqq] (6.,0.)-- (8.,0.);
		\draw [line width=1.pt,color=yqqqyq] (4.,0.)-- (2.,0.);
		\draw [line width=1.pt,color=ffqqqq] (2.,0.)-- (0.,0.);
		\draw [line width=1.pt,color=yqqqyq] (8.,0.)-- (10.,0.);
		\draw [line width=1.pt,color=qqqqff] (10.,0.)-- (12.,0.);
		
		\draw [color=qqqqff](4.6,-0.2) node[anchor=north west] {\textrm{blue}};
		\draw [color=yqqqyq](2.35,-0.2) node[anchor=north west] {\textrm{purple}};
		\draw [color=qqqqff](10.6,-0.2) node[anchor=north west] {\textrm{blue}};
		\draw [color=yqqqyq](8.35,-0.2) node[anchor=north west] {\textrm{purple}};
		\draw [color=ffqqqq](6.7,-0.2) node[anchor=north west] {\textrm{red}};
		\draw [color=ffqqqq](0.7,-0.2) node[anchor=north west] {\textrm{red}};
		\begin{scriptsize}
		\draw [color=black] (4.,0.) circle (2pt);
		\draw [color=black] (6.,0.) circle (2pt);
		\draw [color=black] (8.,0.) circle (2pt);
		\draw [fill=black] (2.,0.) circle (2pt);
		\draw [color=black] (0.,0.) circle (2pt);
		\draw [color=black] (10.,0.) circle (2pt);
		\draw [color=black] (12.,0.) circle (2pt);
		\end{scriptsize}
		\end{tikzpicture}
		\caption{Path between standard purple edges at distance 2.}
		\label{figure:local_picture_2}
	\end{figure}
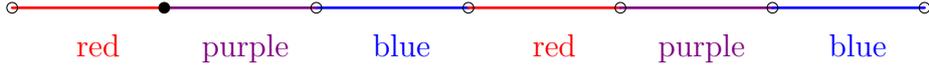 

	So swapping the colors on the path between the purple edges results in them having neighbors of the same color, and then we can recolor them red or blue according to the color of their neighbors. The case of $n=2$ is illustrated in Figures \ref{figure:local_picture_2} and \ref{figure:local_picture_better_2}.   
	
	\begin{figure}[h]
		\centering
		\definecolor{yqqqyq}{rgb}{0.5019607843137255,0.,0.5019607843137255}
		\definecolor{ffqqqq}{rgb}{1.,0.,0.}
		\definecolor{qqqqff}{rgb}{0.,0.,1.}
		\begin{tikzpicture}[line cap=round,line join=round,>=triangle 45,x=1.0cm,y=1.0cm]
		\clip(-0.5,-1.) rectangle (12.5,1.);
		\draw [line width=1.pt,color=ffqqqq] (0.,0.)-- (2.,0.);
		\draw [line width=1.pt,color=qqqqff] (2.,0.)-- (4.,0.);
		\draw [line width=1.pt,color=ffqqqq] (4.,0.)-- (6.,0.);
		\draw [line width=1.pt,color=qqqqff] (6.,0.)-- (8.,0.);
		\draw [line width=1.pt,color=ffqqqq] (8.,0.)-- (10.,0.);
		\draw [line width=1.pt,color=qqqqff] (10.,0.)-- (12.,0.);
		
		\draw [color=ffqqqq](0.7,-0.2) node[anchor=north west] {\textrm{red}};
		\draw [color=qqqqff](2.6,-0.2) node[anchor=north west] {\textrm{blue}};
		\draw [color=ffqqqq](4.7,-0.2) node[anchor=north west] {\textrm{red}};
		\draw [color=qqqqff](6.6,-0.2) node[anchor=north west] {\textrm{blue}};
		\draw [color=ffqqqq](8.7,-0.2) node[anchor=north west] {\textrm{red}};
		\draw [color=qqqqff](10.6,-0.2) node[anchor=north west] {\textrm{blue}};
		
		\begin{scriptsize}
		\draw [color=black] (4.,0.) circle (2pt);
		\draw [color=black] (6.,0.) circle (2pt);
		\draw [color=black] (8.,0.) circle (2pt);
		\draw [fill=black] (2.,0.) circle (2pt);
		\draw [color=black] (0.,0.) circle (2pt);
		\draw [color=black] (10.,0.) circle (2pt);
		\draw [color=black] (12.,0.) circle (2pt);
		\end{scriptsize}
		\end{tikzpicture}
		\caption{Path recolored.}
		\label{figure:local_picture_better_2}
	\end{figure}
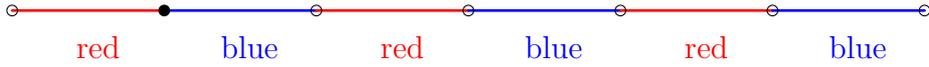  

	As before, we do not loose measurability. Standard purple edges at distance $n$ can be identified using neighborhoods of points, and a sufficiently sparse vertex labeling makes sure that we do not try to recolor the same edges with different colors. After finitely many recoloring steps we finish phase $n$. 
	
	We execute phases $0,1, \ldots, r$, and estimate the measure of purple edges at the end. Purple edges are either standard, or have an endpoint of degree 3. We divide the set $P = \{e \in E ~|~ c(e)=\textrm{"purple"}\}$ accordingly: $P_1=\{e \in P ~|~ e \textrm{ is standard}\}$, $P_2=\{ e=(x,y) \in P ~|~ \deg(x)=3 \textrm{ or } \deg(y)=3\}$.
	
	As points of degree 3 are $r$-sparse in infinite components they have measure at most $2/r$. Each such point has exactly one edge from $P_2$, so $\widetilde{\nu}(P_2) \leq 2/r$. Edges in $P_1$ are also $r$-sparse, so $\widetilde{\nu}(P_1) \leq 2/r$. We indeed proved that $\widetilde{\nu}(P) \leq 4/r$.	 
\end{proof}
	
\medskip	
	
\begin{proof}[Proof of Theorem \ref{theorem:almost_proper_coloring}]
	The idea is to use Theorem \ref{theorem:CSLP_induction} $(d-2)$ times to remove matchings from the edge set, coloring each one with a unique color. Whenever removing those matchings creates finite components we color those using Lemma \ref{lemma:finite_components_smart_coloring}. In the end we will apply Proposition \ref{proposition:induction_start} to the remaining graphing. 
	
	Instead of re-normalizing every time we split our graphings, we will talk about \emph{density} of $(d+1)$-color edges. This means their $\widetilde{\nu}$-measure divided by the $\nu$-measure of the vertex set of the graphing in question. At each step the density of $(d+1)$-color edges will be at most $\varepsilon$, so in the end their density will still be at most $\varepsilon$, and this coincides with their  $\widetilde{\nu}$-measure. 
	
	To be precise we start by picking $r_2 \geq r_0(3)$ such that $4/r_2 \leq \varepsilon$. Looking to apply Theorem \ref{theorem:CSLP_induction} repeatedly we define $r_i \geq r_1(i, r_{i-1})$ for $i=3, \ldots d$, also assuming $r_{i-1} \leq r_i$.
	
	The graphing $\G$ has no vertices of degree $(d+1)$, so in particular vertices of degree $(d+1)$ are $r_d$-sparse. The finite components can be colored with $[d]$ using Lemma \ref{lemma:finite_components_smart_coloring} (choosing $\rho =0$). So without loss of generality we can assume that all components are infinite. Then $\G$ satisfies the conditions of Theorem \ref{theorem:CSLP_induction}, so we find a measurable matching $M_1 \subseteq E(\G)$, such that the set of vertices of degree $d$ in $\G_1=\G \setminus M_1$ is $r_{d-1}$-sparse. We color $M_1$ with color $d$, and turn our attention to $\G_1$. Notice that all degrees in $\G_1$ are either $d$ or $d-1$.
	
	We claim that all finite components of $\G_1$ that contain a vertex of degree $d$ contain at least two. Assume towards contradiction that there was a finite component $G=(F_1,F_2, E')$, with a unique vertex $x$ of degree $d$. Without loss of generality we can assume $x \in F_1$. Double counting the edges we get $(d-1)|F_1|+1 = (d-1)|F_2|$, which is not possible if $d \geq 3$, as the right hand side is divisible by $(d-1)$, while the left hand side is not. This proves our claim. 
	
	This rules out small components with vertices of degree $d$. We use this to argue that in the finite components vertices of degree $d$ are not only sparse, but have small density as well. Components having degree $d$ vertices have size at least $r_{d-1}+1$, since they have at least two vertices at distance at least $r_{d-1}$. The $r_{d-1}/2$ balls around degree $d$ vertices are disjoint, and have size at least $r_{d-1}/2$ as the component has at least $r_{d-1}+1$ points. So the density of degree $d$ points is at most $2 / r_{d-1}$. 
	
	We split $\G_1$ into the disjoint union $\G_1 = \F_1 \cup \G'_1$, where $\F_1$ is the part with finite components, and $G'_1$ is the one with infinite components. To $\F_1$ we apply Lemma \ref{lemma:finite_components_smart_coloring} with $\rho = 2 / r_{d-1} \leq 4 / r_2 \leq \varepsilon$. We use the colors $\{1, \ldots, d-1\}$ for most edges, and $(d+1)$ with density at most $\varepsilon$.
	
	To $G'_1$ we apply Theorem \ref{theorem:CSLP_induction}, and find a measurable matching $M_2$ such that degree $(d-1)$ vertices in $G_2 = \G'_1 \setminus M_2$ are $r_{d-2}$-sparse. We color $M_2$ with color $d-1$. Notice that all vertices in $\G_2$ are of degree either $d-2$ or $d-1$.
	
	As long as $i<d-2$ we will split the graphing $\G_{i}$ into the finite and infinite component parts $\F_i$ and $G'_{i}$. We argue that degree $d-i+1$ vertices in $\F_i$ are not only sparse, but also of small density. The divisibility argument works, as we have $(d-i)|F_1|+1 = (d-i)|F_2|$ being impossible. We color $\F_i$ using Lemma \ref{lemma:finite_components_smart_coloring} with colors $\{1, \ldots, d-i\}$ and $(d+1)$ used with density at most $2/r_{d-i} \leq 4/r_2 \leq \varepsilon$. Then we apply Theorem \ref{theorem:CSLP_induction} to $\G'_i$, find a matching $M_{i+1}$, color it with $(d-i)$ and set $\G_{i+1}=\G_i \setminus M_{i+1}$. The set of vertices of degree $(d-i)$ in $\G_i$ are $r_{d-i}$-sparse.
	
	After coloring $d-2$ matchings we end up with $\G_{d-2}$. This satisfies the conditions of Proposition \ref{proposition:induction_start}, which finishes the coloring using colors $1,2$ and $(d+1)$, the later with density at most $4/r_2 \leq \varepsilon$. This finishes the proof. 
\end{proof}	

\section{Proof of the main Theorem} \label{section:proof_of_main}

In this section we complete the the proof of Theorem \ref{theorem:main}. Take any unimodular random rooted graph $(G,o)$ with distribution $\mu \in \mathcal{M}(\GoG^{2d}_{\circ})$ that is $2d$-regular with probability one. By Theorem \ref{theorem:balanced_orientation} we choose a graphing $\G=(X,E,\nu)$ representing $\mu$ that can be measurably oriented. The orientation is a measurable way of choosing a terminus of each edge, that is a symmetric measurable map $\textrm{or}: E \to X$ such that for all edges $e=(u,v) \in E$ we have $\textrm{or}(e) \in \{u,v\}$. 

Using the orientation we construct a $d$-regular bipartite graphing, by doubling the vertex set, and placing edges according to their orientation. To be precise let $Y=X \times \{1,2\}$, $X_i= X \times \{i\}$, consider the product measure $\nu'$ on $Y$, that is $\nu' = \nu \times u_{\{1,2\}}$, where $u_{\{1,2\}}$ is the uniform distribution on $\{1,2\}$. We build a bipartite, undirected graphing $\mathcal{H}$ on the space $(X_1, X_2, \nu')$. We set $\big((u,1),(v,2)\big) \in E(\mathcal{H})$ if and only if $(u,v) \in E(\G)$ and $\textrm{or}(u,v) = v$.

We set $1/n$ as the amount of error we allow with our coloring. By Theorem \ref{theorem:almost_proper_coloring} we have a measurable coloring $c^{n}_0:E(\mathcal{H}) \to [d+1]$ using the last color infrequently, i.e.\ $\widetilde{\nu}'\big(c_0^{-1}(d+1)\big)\leq 1/n$. The measure $\widetilde{\nu}'$ is the edge measure on the graphing $\mathcal{H}$. The coloring $c^n_0$ gives a coloring $c^n: E(\G) \to [d+1]$, as every oriented edge of $\G$ is represented as exactly one edge of $H$. For an edge $e=(u,v) \in E(\G)$ let $\textrm{or}^{-}(e)$ denote its source, that is $\textrm{or}^{-}(e)=\{u,v\} \setminus \textrm{or}(e)$. (For loops the terminus and source are the same.) We set $c^n(e) = c^n_0 \big((\textrm{or}^{-}(e),1),(\textrm{or}(e),2)\big)$.

We managed to measurable orient and color the edges of $\G$, although the number of colors is $(d+1)$ instead of the desired $d$. We now look at the connected component of a $\nu$-random point $x$ in $\G$, together with the orientation and coloring of edges, we get a random rooted, oriented and colored graph. Our next two lemmas claim that these are almost Schreier graphs and by passing to a weakly convergent subsequence we get an invariant random Schreier graph.

\begin{lemma} \label{lemma:almost_schreier}
	Let $(G_n,o_n,\mathrm{or}_n, c_n)$ denote the oriented, edge-colored component of a $\nu$-random vertex in $\G_n=(X, \nu, \mathrm{or}, c^n)$. Then
	
	\[\mathbb{P} \big[\textrm{all edges around the root } o_n \textrm{ are colored by } [d]\big]  \geq 1 - \frac{4}{n}.\]
\end{lemma}

\begin{proof}
	Let $Y_{d+1}\subseteq Y$ denote vertices of $\mathcal{H}$ that encounter an edge colored $d+1$. 
	\[Y_{d+1}=\{y \in X_1 \cup X_2 ~|~ \exists y'\sim_{\mathcal{H}} y \textrm{ with } c^n_0\big((y,y')\big)=d+1\}.\] 
	Our assumption of $(d+1)$-colored edges being infrequent implies $\nu(Y_{d+1}) \leq 2/n$.
	
	Similarly, set \[X_{d+1}=\{x \in X ~|~ \exists x'\sim_{\G} x \textrm{ with } c^n\big((x,x')\big)=d+1\}.\] By construction we have $x \in X_{d+1}$ if and only if $(x,1) \in Y_{d+1}$ or $(x,2) \in Y_{d+1}$. This gives $\nu(X_{d+1}) \leq 4/n$.
	
	The set $X \setminus X_{d+1}$ is exactly the vertices that have all edges connected to them colored by $[d]$, so we have	
	\[\mathbb{P} \big[\textrm{all edges around the root } o_n \textrm{ are colored by } [d]\big] = \nu(X \setminus X_{d+1}) \geq 1 - \frac{4}{n}.\]
\end{proof}

By passing to a subsequence we can assume that $(G_n,o_n,\textrm{or}_n,c_n)$ converges to some random rooted, oriented, colored graph. Note that for all $n$ the undirected, uncolored rooted graph $(G_n,o_n)$ is the same as $(G,o)$ in distribution, so the same holds for the limit. We will denote the limit by $(G,o,\textrm{or}, c)$, which highlights the fact that it is a random orientation and coloring of the random rooted graph $(G,o)$ we began with. 

\begin{remark}
	In our notation logically $\mu'_{\G_n}$ should denote the distribution of $(G_n,o_n,\textrm{or}_n,c_n)$, and $\mu'$ the subsequential weak limit. Notice however, that $\mu'_{\G_n} \notin \mathcal{M}(\GoG^{\mathrm{Sch}}_{\circ})$. The weak convergence $\mu'_{\G_n} \to \mu'$ takes place in the larger space $\mathcal{M}(\GoG^{2d,\mathrm{or},c}_{\circ})$. Here $\GoG^{2d,\mathrm{or},c}_{\circ}$ denotes the space of rooted, connected graphs with degree bound $2d$, with a balanced orientation and a coloring of the edges by $d+1$ colors, each vertex having at most one in- and outedge of each color. The forgetting function $\Phi: \GoG^{2d,\mathrm{or},c}_{\circ} \to \GoG^{2d}_{\circ}$ still makes sense, and our construction ensured $\Phi_{*}(\mu'_{(\G_n)}) = \mu$. This implies $\Phi_{*} \mu'=\mu$, which we shorthanded into the notation $(G,o,\textrm{or}, c)$.
	
	A priori we do not know that $\mu' \in \mathcal{M}(\GoG^{\mathrm{Sch}}_{\circ})$, and $F_d$ only partially acts on $\GoG^{2d,\mathrm{or},c}_{\circ}$. This makes the statement of our next lemma somewhat cumbersome.
\end{remark}

\begin{lemma} \label{lemma:invariant_schreier_limit}
	With probability 1 all the edges adjacent to the root in $(G,o,\mathrm{or}, c)$ are colored with $[d]$. For any $s \in [d]$ we define $s.o$ as the unique neighbor $v$ of $o$ with the edge $(o,v)$ colored $s$ and oriented towards $v$. Such a vertex exists almost surely.
	We claim that $(G,o,\mathrm{or},c)$ and $(G,s.o,\mathrm{or},c)$ are the same in distribution.  
\end{lemma}

In the lemma we only claim that edges at the root have the right colors. Together with the invariance however this implies that all edges have the right colors, so Lemma \ref{lemma:invariant_schreier_limit} finishes the proof of Theorem \ref{theorem:main}.

\medskip

\begin{proof}
	The first part of the statement follows easily from Lemma \ref{lemma:almost_schreier}. The probability of seeing the colors $[d]$ adjacent to $o$ in $(G,o,\textrm{or}, c)$ is the limit of the probabilities of the same event for $(G_n,o_n,\textrm{or}_n,c_n)$, which is clearly $1$.
	
	To show invariance with respect to moving the root, let $\alpha$ denote a finite (connected) rooted, oriented, $[d+1]$ colored graph of radius $r$. Assume also that the $2d$ edges at the root are properly oriented an colored.  
	
	We have to show \[\mathbb{P}_{(G,o,\mathrm{or},c)}\big[B_G(r,o) \cong \alpha\big] = \mathbb{P}_{(G,s.o,\mathrm{or},c)}\big[B_G(r,s.o) \cong \alpha\big].\]
	
	The event $B_G(r,s.o) \cong \alpha$ can be seen by the $r+1$ neighborhood of the root in $(G,o)$. So we collect all connected, rooted, oriented and $[d+1]$-colored finite graphs $\beta=(F_{\beta},o_{\beta},\mathrm{or}_{\beta}, c_{\beta})$ of radius at most $r+1$, where $B_{F_{\beta}}(r,s.o_{\beta}) \cong \alpha$. We implicitly assumed that $s.o_{\beta}$ makes sense. Let $\mathbf{B}$ denote the set of all such $\beta$. Now we can express the right hand side of the above as
	
	\[\mathbb{P}_{(G,s.o,\mathrm{or},c)}\big[B_G(r,s.o) \cong \alpha\big] = \mathbb{P}_{(G,o,\mathrm{or},c)}\big[B_G(r+1,o) \in \mathbf{B}\big].\]
	
	Since these events correspond to clopen sets in the space of rooted, oriented, colored graphs we have
	
	\[\mathbb{P}_{(G,o,\mathrm{or},c)}\big[B_G(r,o) \cong \alpha\big] = \lim_{n \to \infty} \mathbb{P}_{(G_n,o_n,\mathrm{or}_n,c_n)}\big[B_{G_n}(r,o_n) \cong \alpha\big],\]
	\[\mathbb{P}_{(G,o,\mathrm{or},c)}\big[B_G(r+1,o) \in \mathbf{B}\big] =
	\lim_{n \to \infty}
	\mathbb{P}_{(G_n,o_n,\mathrm{or}_n,c_n)}\big[B_{G_n}(r+1,o_n) \in \mathbf{B}\big].\]
	
	To compare the events $B_{G_n}(r,o_n) \cong \alpha$ and $B_{G_n}(r+1,o_n) \in \mathbf{B}$ we introduce the sets 
	
	\[X^{\alpha}_{n}=\{x \in X ~|~ B_{(\G,c^n)}(r,x) \cong \alpha\},\]
	\[X^s_n:\{x \in X ~|~ \textrm{ there is an outward $s$-edge of $(\G,\mathrm{or},c^n)$ at } x\},\]
	\[X^{\mathbf{B}}_{n}=\{x \in X ~|~ B_{(\G,c^n)}(r+1,x) \in \mathbf{B}\}.\]

	We also introduce the partial bijection $\varphi^s_n$ defined on $X^s_n$, setting $\varphi^s_n(x) = y$ if $(x,y)$ is the $s$-colored edge of $(\G,\mathrm{or},c^n)$ oriented towards $y$. As $\G$ is a graphing, this partial bijection is measure preserving from $X^s_n$ to $\varphi^s_n(X^s_n)$. Also $X^{\mathbf{B}}_n = \{x \in X^s_n ~|~ \varphi^s_n(x) \in X^{\alpha}_n\}$, so
	
	\[\mu(X^{\mathbf{B}}_n) = \mu\big((\varphi^s_n)^{-1}(\mathrm{ran}\varphi^s_n \cap X^{\alpha}_n  \big) = \mu(\mathrm{ran}\varphi^s_n \cap X^{\alpha}_n).\]
	We can rewrite the probabilities as follows:
	
	\[\mathbb{P}_{(G_n,o_n,\mathrm{or}_n,c_n)}\big[B_{G_n}(r,o_n) \cong \alpha\big] = \mu(X^{\alpha}_n),\]
	\[\mathbb{P}_{(G_n,o_n,\mathrm{or}_n,c_n)}\big[B_{G_n}(r+1,o_n) \in \mathbf{B}\big] = \mu(X^{\mathbf{B}}_n) = \mu(\mathrm{ran}\varphi^s_n \cap X^{\alpha}_n).\]
	We have $\mu(X^{\alpha}_n) - \mu(\mathrm{ran}\varphi^s_n \cap X^{\alpha}_n) \leq 1- \mu(\mathrm{ran}\varphi^s_n) = 1- \mu(X^s_n) \to 0$, so
	
	\[\lim_{n \to \infty} \mathbb{P}_{(G_n,o_n,\mathrm{or}_n,c_n)}\big[B_{G_n}(r,o_n) \cong \alpha\big] = \lim_{n \to \infty}
	\mathbb{P}_{(G_n,o_n,\mathrm{or}_n,c_n)}\big[B_{G_n}(r+1,o_n) \in \mathbf{B}\big].\]
	By the equalities above this implies
	\[\mathbb{P}_{(G,o,\mathrm{or},c)}\big[B_G(r,o) \cong \alpha\big] = \mathbb{P}_{(G,o,\mathrm{or},c)}\big[B_G(r+1,o) \in \mathbf{B}\big] = \mathbb{P}_{(G,s.o,\mathrm{or},c)}\big[B_G(r,s.o) \cong \alpha\big],\]
	which finishes the proof.
\end{proof}

\section{Graphings from $F_d$ actions} \label{section:graphings}

In this section we consider probability measure preserving (p.m.p.) actions of $F_d$, and prove Corollary \ref{corollary:graphing_lift}.

Given a p.m.p.\ action of $F_d$ on a standard Borel probability space $(X, \nu)$ the associated Schreier graph $\Sch(F_d \acts X, S)$ with the standard generating set $S$ is a $2d$-regular graphing. The action of $F_d$ provides a measurable orientation and coloring of the edges. 

Not all $2d$-regular graphings come from p.m.p. actions of $F_d$. In fact finding a measurable orientation and coloring for the edges of a $2d$-regular graphing $\G=(X,E, \nu)$ is equivalent to finding such a p.m.p.\ action $F_d \acts (X, \nu)$ with $\Sch(F_d \acts X, S)=\G$ as unlabeled graphings.

When two graphings $\G_1$ and $\G_2$ represent the same unimodular random rooted graph, that is $\mu_{\G_1} = \mu_{\G_2}$ we say that they are \emph{locally equivalent}. Recall that a graphing $\G_1=(X_1, E_1, \nu_1)$ is the \emph{local isomorphic image} of another graphing $\G_2=(X_2, E_2, \nu_2)$ if their is a measure preserving map $\varphi: X_2 \to X_1$ such that $\varphi_{*}\nu_2=\nu_1$, and  $\big(C_{\G_2}(x),x\big) \cong \big(C_{\G_1}(\varphi(x)),\varphi(x)\big)$ for $\nu_2$-almost all $x \in X_2$.
It is clear that if $\G_1$ is a local isomorphic image of $\G_2$, then they are locally equivalent. However, local isomorphisms are not invertible in general. To this end we introduce a symmetric relation, \emph{bi-local isomorphism}. We say $\G_1$ and $\G_2$ are bi-locally isomorphic if they are both local isomorphic images of some graphing $\G_3$.

Clearly, bi-local isomorphism implies local equivalence. We will exploit the fact that the converse also holds. 

\begin{theorem}\label{theorem:bilocal}
	Two graphings are locally equivalent if and only if they are bi-locally isomorphic.
\end{theorem}

See \cite[Theorem 18.59]{lovasz2012large} for a detailed proof. Most of the argument is also present in \cite{hatami2014limits}.

\medskip

\begin{proof}[Proof of Corollary \ref{corollary:graphing_lift}]
	Let $\G$ be an arbitrary $2d$-regular graphing. By Theorem \ref{theorem:main} the corresponding unimodular random rooted graph $\mu_{\G}$ has an invariant random Schreier decoration $\mu'$. We use $\mu'$ to build a graphing that is locally equivalent to $\G$ with measurable orientation and coloring of the edges. 
	
	Let $\GoG_{\circ}^{\Sch,l}$ denote the space of Schreier graphs of $F_d$ with vertices labeled by $[0,1]$. The Borel structure and the edges are define as in Subsection \ref{subsection:constructing_graphings}. By choosing uniform i.i.d.\ vertex labels for a $\mu'$-random Schreier graph, we get a measure $\bar{\mu}$ on $\GoG_{\circ}^{\Sch,l}$ that turns it into a graphing locally equivalent to $\G$. The graphing property holds because $\mu'$ is invariant, and $(\GoG_{\circ}^{\Sch,l}, \bar{\mu})$ is locally equivalent to $\G$ because $\Phi_{*}\bar{\mu}=\Phi_{*}\mu'=\mu_{\G}$. It is clear that $(\GoG_{\circ}^{\Sch,l}, \bar{\mu})$ has a measurable orientation and coloring, as by construction it carries a p.m.p.\ action of $F_d$.
	
	By Theorem \ref{theorem:bilocal} $\G$ and $(\GoG_{\circ}^{\Sch,l}, \bar{\mu})$ are bi-locally isomorphic, so they are both local isomorphic images of some graphing  $\G'=(X', E', \nu')$. The measurable orientation and coloring can be pulled back by the local isomorphism $\varphi: X' \to \GoG_{\circ}^{\Sch,l}$, so $\G'$ satisfies the requirements of the corollary.
\end{proof}

\subsection{Unimodular decorations}

The argument in the previous subsection can be generalized to arbitrary decorations of graphs. To define an abstract decoration, let $D_1, \ldots D_n$ be compact standard Borel spaces serving as labels. This allows finite sets as well. We use $D_i$ to label $k_i$-tuples of vertices.

A \emph{decoration} of the graph $G=(V,E)$ is the finite set of functions $\mathcal{C}=\{c_1,\ldots, c_n\}$ with $c_i:V(G)^{k_i} \to D_i$. The decorated graph consists of the pair $(G,\mathcal{C})$.

Fix the sets of labels $\mathcal{D}= (D_1, \dots, D_n)$, the size of tuples $\mathbf{k}=(k_1, \ldots, k_n)$ and a degree bound $\Delta$. The space $\GoG_{\circ}^{\Delta, \mathcal{D}, k}$ of rooted, connected, decorated graphs $(G,o,\mathcal{C})$ with degree bound by $\Delta$ is a compact standard Borel space as before. 

The Borel structure is generated by the following cylinder sets. For any $r \geq 0$ we fix the isomorphism type of the ball with radius $r$ about the root, and also for every $k_i$-tuple of vertices in the ball we specify a Borel set in $D_i$ from which its label is to be chosen.

Unimodular measures $\mu \in \mathcal{M}(\GoG_{\circ}^{\Delta, \mathcal{D}, k})$ can be defined exactly as before, by the involution invariance of the bi-rooted measure $\widetilde{\mu} \in \mathcal{M}(\GoG_{\circ\circ}^{\Delta, \mathcal{D}, k})$ obtained by taking a step with the simple random walk on a $\mu$-random sample of $\GoG_{\circ}^{\Delta, \mathcal{D}, k}$.

\begin{definition} \label{definition:unimodular_decoration}
	Let $\mu \in \GoG_{\circ}^{\Delta}$ be a unimodular random rooted graph. A \emph{unimodular random decoration} of $\mu$ is a measure $\mu' \in \mathcal{M}(\GoG_{\circ}^{\Delta, \mathcal{D}, k})$ that is unimodular, and $\Phi_{*}\mu' = \mu$. 
\end{definition}

This notion has been studied in the context of Schreier decorations by Biringer and Tamuz in \cite{biringer2017unimodularity}. They show that for a random rooted Schreier graph unimodularity is equivalent to invariance. Cannizzo also proves this in \cite{cannizzo2013invariant}. The equivalence shows how Definition \ref{definition:unimodular_decoration} is a generalization of Definition \ref{definition:invariant_schreierization} in the Introduction.

When decorating a graphing $\G$ the only formal difference is requiring measurability of the maps $c_i:V(\G)^{k_i} \to D_i$. Each connected component becomes a decorated graph when restricting the decoration. 

Since we are only interested in the connected components, the values of the $c_i$ only matter on $k_i$-tuples where the vertices are in the same component. As the graph is measurable, modifying the values on tuples that are not contained in connected components does not change measurability of the $c_i$. So we can assume the $c_i$ to be constant on such tuples. 

The following theorem translates between the two languages.

\begin{theorem} \label{theorem:translation}
	Let $P$ be a property of rooted graphs, and $Q$ be a property of rooted, decorated graphs. The following are equivalent.
	\begin{enumerate}[i)]
		\item Every unimodular random rooted graph $\mu$ that has property $P$ almost surely has a unimodular random decoration $\mu'$ that has property $Q$ almost surely.
		\item Every graphing $\G$ with almost every component having property $P$ is the local isomorphic image of a graphing $\G'$ that has a measurable decoration with almost every component having property $Q$. 
	\end{enumerate}
\end{theorem}

\begin{proof}
	The $ii)$ implies $i)$ part is the easy one. We pick $\G$ to represent $\mu$, then $\G'$ also represents $\mu$. Together with the measurable decoration it defines $\mu'=\mu'_{(\G,\mathcal{C})}$.
	
	The $i)$ implies $ii)$ part is proved exactly like Corollary \ref{corollary:graphing_lift}. Starting from $\G$ we use $i)$ for $\mu_{\G}$ to find a locally equivalent graphing $\G_0$ that can be measurable decorated, and use Theorem \ref{theorem:bilocal} to find $\G'$ that has local isomorphisms into both. The decoration can be pulled back from $\G_0$ to $\G'$.
\end{proof}

\medskip

We now proceed to prove Corollary \ref{corollary:proper_coloring_lift}.

\medskip

\begin{proof}[Proof of Corollary \ref{corollary:proper_coloring_lift}]
	Our idea is to show that the unimodular random graph $\mu_{\G}$ has a unimodular random proper edge coloring with $d$ colors, and then use the correspondence in Theorem \ref{theorem:translation}.
	
	First we complement our bipartite graphing $\G$ with additional vertices and edges, such that it becomes $d$-regular. Then we color the edges using \ref{theorem:almost_proper_coloring} with smaller and smaller weight of $(d+1)$-color edges. The coloring $c_n$ restricted to $\G$ give unimodular random colored graphs $\mu_{(\G,c_n)}$ when looking at the component of a random point. A subsequential weak limit will be a unimodular random edge coloring of $\mu_{\G}$.
	
	We also want to make sure that the graphing we get in the end is bipartite. This is equivalent to the existence of a measurable proper vertex labeling with 2 labels. So we start with such a labeling $l: V(G) \to \{1,2\}$, and retain the information in the random graph, that is we consider $\mu_{(\G,c_n, l)}$. Now a subsequential weak limit will still be unimodular random, carrying both edge and vertex decoration. This makes sure that the graphing $G'$ that has $G$ as a local isomorphic image is both bipartite and the edges can be measurably colored. 
\end{proof}

\bibliographystyle{alpha}
\bibliography{refs}

\end{document}